\documentclass[a4paper,12pt]{amsart}
\usepackage{amsmath}
 \usepackage{latexsym, amsthm, amscd, euscript, float, subfig}
\usepackage{amssymb}
\usepackage{ifthen}
\usepackage{graphicx}
\usepackage{float}
\usepackage[utf8]{inputenc}
\usepackage{cite}
\usepackage{amsfonts}
\usepackage{amscd}
\usepackage{amsxtra}
\usepackage{color}
\usepackage{mathtools}
\usepackage[dvipsnames]{xcolor}

\newtheorem{theorem}{Theorem}[section]
\newtheorem{lemma}[theorem]{Lemma}

\newtheorem{definition}[theorem]{Definition}

\theoremstyle{definition}

\numberwithin{equation}{section}

\def\be{\begin{equation}}
\def\ee{\end{equation}}

\def\diam{\text{ diam}}
\def\dist{\text{ dist}}

\def\ra{\rightarrow}
\def\cra{\curvearrowright}
\def\ds{\text{ ds}}
\def\dt{\text{ dt}}
\newcounter{alphabet}

\makeatletter
\newcommand*{\rom}[1]{\expandafter\@slowromancap\romannumeral #1@}
\makeatother

\begin{document}
\bibliographystyle{amsplain}
\title{Uniformization of intrinsic Gromov hyperbolic spaces}
\author[Vasudevarao Allu]{Vasudevarao Allu}
\address{Vasudevarao Allu,Department of Mathematics, School of Basic Sciences, 
Indian Institute of Technology Bhubaneswar,
Bhubaneswar-752050, 
Odisha, India.}
\email{avrao@iitbbs.ac.in}
\author[Alan P Jose]{Alan P Jose}
\address{Alan P Jose, Department of Mathematics, School of Basic Sciences, Indian Institute of Technology Bhubaneswar,
Bhubaneswar-752050, Odisha, India.}
\email{alanpjose@gmail.com}

\begin{abstract}
 The purpose of this paper is to provide a uniformization procedure for Gromov hyperbolic spaces, which need not be geodesic or proper. 
We prove that the conformal deformation of a Gromov hyperbolic space is a bounded uniform space. 
Further, we show that there is a natural quasi-isometry between the Gromov boundary and the metric boundary of the deformed space. 
Our main results are a generalization of the results of Bonk, Heninonen, and Koskela [Proposition 4.5, Proposition 4.13, Ast\'{e}risque 270 (2001)].
\end{abstract}

\maketitle
\section{Introduction}
Around 1980's, Mikhail Gromov introduced the notion of "negative curvature" for a general metric space in the sense of coarse geometry, which is now known as Gromov hyperbolic spaces. 
Even though, the concept was conceived in the setting of geometric group theory, it has played an increasing role in the analysis of general metric spaces. 
Bonk, Heinonen, and Koskela \cite{BHK} have established a fundamental two-way correspondence between uniform metric spaces and proper geodesic Gromov hyperbolic spaces. 
In the context of Euclidean spaces, this result can be seen as a generalization of the Riemann mapping theorem in 
$\mathbb{C}$, as the unit disk is a uniform domain. 
Thus, in this sense, Gromov hyperbolic spaces can be viewed as an appropriate counterpart to simply connected domains in the plane. 

Inspired from the well-known properties of a hyperbolic geodesic in the unit disk, Martio and Sarvas \cite{martio_sarvas_} have introduced the class of uniform domains in Euclidean spaces. 
Bonk, Heinonen, and Koskela \cite{BHK} have introduced the concept of uniform metric spaces and established the relationship between Gromov hyperbolicity and uniform metric spaces. They obtained the following result:
\begin{theorem}\label{bhk_result}
The conformal deformations $X_\epsilon = \left(X, d_\epsilon\right)$ of a proper, geodesic $\delta$-hyperbolic space $X$ are bounded $A(\delta)-$uniform spaces for $0<\epsilon\leq \epsilon_0(\delta)$.
\end{theorem}
While their work focused on locally compact metric spaces, V{\"a}is{\"a}l{\"a} \cite{vaisala_2004_h} has extended a similar theory to infinite-dimensional Banach spaces, where local compactness is not available. 
He has worked extensively on the theory of quasiconformal mappings in infinite dimensional Banach spaces where quasihyperbolic geodesics may not exist (see \cite{quasiworld}). 
But such domains equipped with quasihyperbolic metric were intrinsic spaces, which means that the distance between two points is equal to the infimum of the length of all arcs joining these points. 
Motivated by this, V{\"a}is{\"a}l{\"a} has developed a theory of Gromov hyperbolic spaces that need not be geodesic or proper (see \cite{vaisala_2004_gh}). 

It is worth to point out that Butler \cite{butler} and Zhou et. al \cite{zpr} have constructed an unbounded analogue of Theorem \ref{bhk_result} using densities that are exponential in Busemann functions.
Zhou et. al \cite{zpr} showed that there is a one-to-one correspondence between bilipschitz classes or proper geodesic Gromov hyperbolic spaces that are roughly starlike with respect to a point on Gromov boundary and the quasisimilarity classes of unbounded locally compact uniform spaces. 
Butler  \cite{butler} constructed an unbounded analogue of Theorem \ref{bhk_result} for  geodesic spaces that are roughly starlike with respect to a point  in its Gromov boundary.
In an earlier work we have constructed a uniformization technique for complete, intrinsic Gromov hyperbolic spaces using densities which are exponential in Busemann functions. This paper is the result of our desire to know whether there is an analogue of Theorem \ref{bhk_result} in the setting of intrinsic Gromov hyperbolic spaces where the existence of geodesics are not assured. Our main result is the following.
\begin{theorem}\label{uniformization}
Let $X$ be an intrinsic Gromov hyperbolic space. Then the conformal deformations $X_\epsilon = \left(X, d_\epsilon\right)$ are bounded $A-$ uniform spaces for $0<\epsilon \leq \epsilon_0$.
\end{theorem}
In \cite[Proposition 4.13]{BHK}, Bonk, Heinonen, and Koskela have established that there is a natural quasi-ismometric identification between the Gromov boundary of a proper, geodesic, hyperbolic space and the metric boundary of the uniformized space $X_\epsilon$. We provide an analogous result as follows. We denote the Gromov boundary of a hyperbolic space $X$ as $\delta_\infty X$.
\begin{theorem}\label{thm2}
Let $X$ be a $\delta-$hyperbolic space and $\epsilon_0$ be the constant as in obtained in Theorem \ref{allu_jose_02_0002} such that $\epsilon_0 <\min\left\{1, 1/5\delta\right\}$. Then there is a natural identification which is a  quasi-isometry between  $\left(\delta_\infty X, \theta_{\epsilon, o}\right) $ and  $\left(\delta_\epsilon X, d_\epsilon\right)$.
\end{theorem}
Because of the lack of geodesics, V{\"a}is{\"a}l{\"a} developed the tools of $h-$short arcs, roads, and biroads instead of geodesics, geodesic rays and geodesic lines respectively. 
Note that the theory of Gromov hyperbolic spaces has been widely explored but there is always the underlying assumption that the spaces are geodesic and proper. 
But, it is worth to point out that Bonk and Schramm \cite{bs} have introduced the concepts of almost geodesic spaces, rough geodesics and rough geodesic rays and they proved that every hyperbolic space can be embedded isometrically into a complete geodesic hyperbolic space. 

In the study of general Gromov hyperbolic spaces the proofs are more complicated than in the classical case. A reader can question the motivation to study these theory.
We note that to assure the existence of geodesic rays or lines one may have to make additional assumptions, for example, that the space is proper. 
Furthermore, considering general hyperbolic spaces is more suited to the philosophy of Gromov hyperbolic spaces since bounded distortions does not affect the structure of the space on large scale.
\section{Preliminaries}
\subsection{Notations}
Throughout this paper, let  $X$ denote an arbitrary metric space and for $x, y\in X$ we denote $|x-y|$ as the distance between $x$ and $y$. 
Every metric space $X$ can be isometrically embedded into a complete metric space which we denote as $\bar{X}$. 
In this setting, we denote $\bar{A}$ as $\delta A$ the closure and boundary of $A$ in the metric topology.
We call $\delta X:= \bar{X}\setminus X$ as the metric boundary of $X$. For $x\in X$, we write $\delta_X(x):= \dist \left(x, \delta X\right)$ for the distance between $x$ and the boundary of $X$, often, if there is no possibility of confusion, we denote it as $\delta(x)$.

\subsection{Arcs in metric space}
An arc in a space $X$ is a subset homeomorphic to a closed real interval which implies that the arc is compact and has two endpoints. We write $\alpha :x \cra y$ to denote that $\alpha $ is an arc with endpoints $x$ and $y$. Often, this notation also provides the direction for $\alpha$ from $x$ to $y$. For an arc $\alpha :x \cra y$, we denote $\alpha[u, v]$ as the closed sub-arc of $\alpha$ between the points $u, v \in \alpha$, where $u\in \alpha[x, v]$. For an arc $\alpha : [a, b] \rightarrow X$, we sometimes denote $\alpha(s, t)$ as the image of the function $\alpha$ restricted to the interval $[s, t]\subset [a, b]$.

For an arc $\alpha : [a, b] \rightarrow X$, we define the length of the arc $\alpha $ in the usual way as
$$l(\alpha) := \sup \left\{ \sum_{i=1}^n |\alpha(t_i) - \alpha(t_{i-1})| \middle|  a= t_0<t_1<...<t_n=b\right\},$$
and the arc $\alpha$ is rectifiable if $l(\alpha) <\infty$ and the space $X$ is rectifiably connected if each pair of points can be joined by a rectifiable path. A metric space $X $ is said to be \textit{intrinsic }if 
$$|x-y|=\inf\left\{l(\alpha)|\alpha: x\cra y\right\},$$
for all $x, y\in X$. Let $h\geq 0$. An arc $\alpha:x\cra y$ is \textit{$h-$short} if
$$l(\alpha)\leq |x-y|+h.$$
It is easy to see that every subarc of an $h-$short arc is again $h-$short. For any rectifiable arc $\alpha $ in $X$ there exists an arclength parametrization $\alpha:[0, L] \rightarrow X$, where $L= l(\alpha)$. We may assume that the arcs considered in this paper are arclength parametrized without further mention. It is known that if $\psi: [0, L] \rightarrow \alpha$ is the arclength parametrization of an $h-$short arc $\alpha$, where $L= l(\alpha)$, then
\be \label{hroughgeo}
|s-t|-h \leq |\psi(s) -\psi(t)| \leq |s-t|
\ee
for all $s, t \in [0, L]$.\\[2mm]
The idea of rough geodesic ray and rough geodesic segments have been introduced by Bonk and Schramm\cite{bs}. An arc $\alpha: [a, b] \rightarrow X$ is said to be $k-$rough geodesic if it satisfies 
$$|s-t|-k \leq |\alpha(s) -\alpha(t)| \leq |s-t|+k,$$
for every $s, t\in [a, b]$. Thus an $h-$ short arc is $h-$rough geodesic. Note that the converse is not true, in fact a rough geodesic need not be even continuous.
A map $f:(X, d)\rightarrow (Y, d')$ between metric spaces is said to be $A$-quasi-isometry $(A\geq 1)$ or simply quasi-isometry if 
$$d(x, y)/A \leq d'(f(x), f(y)) \leq A d(x, y)$$
for every $x, y\in X$.
\subsection{Conformal deformations}
Let $X$ be a rectifiably connected metric space, we fix a point $p\in X$ and  define a family of maps $\rho_\epsilon :X \rightarrow (0, \infty)$ for $\epsilon >0$ as
\begin{equation}\label{density}
\rho_\epsilon (x) = \exp\left(-\epsilon|x-p|\right).
\end{equation}
Define a metric $d_\epsilon$ on $X$ by 
\begin{equation}\label{deformedmetric}
d_\epsilon(x, y) = \inf \int_\gamma \rho_\epsilon \ds ,
\end{equation}
where the infimum is taken over all rectifiable curves $\gamma $ joining $x$ and $y$. We write  $X_\epsilon = \left(X, d_\epsilon \right)$ for conformal deformation of $X$ with a conformal factor of $\rho_\epsilon$.                                                                                                                                                                                                                                                                                                                                                                                                                        Note that, for any rectifiable curve $\gamma$ in $X$, the length of the curve $\gamma$, denoted by $l_\epsilon\left(\gamma\right)$, in the metric space $X_\epsilon$ is given by
$$l_\epsilon(\gamma)= \int_\gamma \rho_\epsilon \ds,$$
provided the identity map $(X, |x-y|) \rightarrow (X, l)$ is a homeomorphism, where $l$ is the inner metric of $X$(see \cite[Lemma 2.6]{BHK} and Appendix \cite{BHK} for a detailed discussion).\\[2mm]
Applying triangle inequality, we derive the following Harnack's type inequality:
\begin{equation}\label{harnack}
\exp\left(-\epsilon|x-y|\right) \leq \frac{\rho_\epsilon(x)}{\rho_\epsilon(y)}\leq \exp\left(\epsilon|x-y|\right)
\end{equation}
for every $x,y\in X$ and $\epsilon >0$.
\subsection{Uniform spaces}
Uniform domains in the Euclidean spaces was introduced by Martio and Sarvas \cite{martio_sarvas_}, and these domains became the so called 'nice' domains in the theory of quasiconformality. 
Roughly speaking, a domain is called uniform if any two points can be joined by an arc which is not too close to the boundary or not too long. 
\begin{definition} \label{uniform}
Let $X$ be a rectifiably connected incomplete metric space and let $A\geq 1$, then $X$ is called $A-$ uniform space if for any two points $x, y\in X$ there exists a rectifiable arc $\gamma:x\cra y$ satisfying the following conditions.
\begin{enumerate}
\item (Quasiconvex condition) $l(\gamma) \leq A|x-y|$,
\item (Double cone condition) $\min \left\{l(\left(\gamma[x, z]\right), l(\left(\gamma[z, y]\right)\right\}\leq A \delta(z)$.
\end{enumerate}
\end{definition}

\subsection{Gromov hyperbolic spaces}
There has been many interesting studies in the theory of Gromov hyperbolic spaces. For further read, we refer to \cite{bridson, buyalo}. We follow  V{\"a}is{\"a}l{\"a} \cite{vaisala_2004_gh} in defining general Gromov hyperbolic spaces.
\begin{definition}[Gromov hyperbolic space]
For $x, y, p\in X$ we define the \textit{Gromov product} $(x|y)_p$ by 
$$2(x|y)_p = |x-p|+|y-p|-|x-y|.$$
Let $\delta \geq 0$. A space is \textit{Gromov $\delta-$hyperbolic} if 
$$(x|z)_p \geq \min\left\{(x|y)_p , (y|z)_p\right\} -\delta$$
for all $x, y, z, p\in X$. A space is \textit{hyperbolic} if it is Gromov $\delta-$hyperbolic for some $\delta\geq 0$.
\end{definition}
The Gromov product $(x|y)_p$ measures the error in the triangle inequality.
V{\"a}is{\"a}l{\"a} \cite{vaisala_2004_h} has proved that uniform domains equipped with the quasihyperbolic metric in arbitrary Banach spaces are hyperbolic. 
Bounded spaces are trivially hyperbolic hence we make the standing assumption that every Gromov hyperbolic spaces we consider in this paper are unbounded.
\begin{definition}[$h-$short triangles]\label{hshort}
By an $h-$short triangle we mean a triple of $h-$short arc $\alpha : y\cra z$, $\beta : x\cra z$, $\gamma: x\cra y$. 
The points $x, y, z$ are the vertices and the arcs $\alpha, \beta, \gamma$ are the sides of the triangle $\Delta=\left(\alpha, \beta, \gamma\right)$. 
\end{definition}
We follow the notations used by V{\"a}is{\"a}l{\"a} \cite{vaisala_2004_gh}. Let $\Delta=\left(\alpha, \beta, \gamma\right)$ be an $h-$short triangle in a hyperbolic space as in  the Definition \ref{hshort}.
Let $x_\gamma, y_\gamma$ be the points in $\gamma $ such that $l(\gamma_x)= (y|z)_x$ and $l(\gamma_y)= (x|z)_y$, where $\gamma_x = \gamma[x, x_\gamma]$ and $\gamma_y = \gamma[y_\gamma, y]$. 
If we set $\gamma^*= \gamma[x_\gamma, y_\gamma]$ which is called the center of the side $\gamma$ in the triangle $\Delta$, then $\gamma$ is the union of three subarcs $\gamma_x, \gamma^*, \gamma_y$. We call it the subdivision of $\gamma$ induced by the triangle $\Delta$ or by the point $z$. Similarly, we can divide each of the sides into disjoint union of three arcs. \\[2mm]
Tripod map is an important tool in the theory of geodesic Gromov hyperbolic spaces.  V{\"a}is{\"a}l{\"a} \cite[Tripod lemma 2.15]{vaisala_2004_gh} obtained the following result which is sufficient for most of the applications.
\begin{lemma} \label{allu_jose_02_0013}
Suppose that $\alpha_i :a\cra b_i, i=1,2,$ are $h-$short arcs in a $\delta-$hyperbolic space. Let $x_1\in \alpha_1$ be a point with $|x_1 -a| \leq (b_1|b_2)_a$, and let $x_2, x_2' \in \alpha_2$ be points with $|x_2-a|=|x_1-a|$ and $l\left(\alpha_2[a, x_2']\right) = l\left(\alpha_1[a, x_1]\right)$. Then
$$|x_1-x_2| \leq 4\delta +h, |x_1-x_2'| \leq 4\delta+2h.$$
\end{lemma}
Next, we give the definition of Gromov boundary. Note that, in a proper geodesic hyperbolic metric space one could define the Gromov boundary as the equivalence classes of geodesic rays. 
But since we are interested in spaces which are neither geodesic nor proper we will adapt the definition of V{\"a}is{\"a}l{\"a} \cite{vaisala_2004_gh} using the Gromov sequences. 
Note, that both the definitions are equivalent in geodesic proper hyperbolic space.
\begin{definition}[Gromov sequences]
Let $X$ be a metric space and let $o\in X$ be fixed. We say that a sequence $\{x_n\} $ in $X$ is a Gromov sequence if $(x_i|x_j)_o \rightarrow \infty$ as $i\rightarrow \infty$ and $j\rightarrow \infty$. 
\end{definition}
Two Gromov sequences $\{x_n\}$ and $\{y_n\}$ are said to be equivalent if $(x_i|y_i)_o \rightarrow \infty $ as $i\rightarrow \infty$ and in a hyperbolic space this relation is an equivalence relation.
Let $\hat{x}$ be the equivalence class containing the Gromov sequence $\{x_n\}$. The Gromov boundary $\delta_\infty X $ is defined as 
\be \label{allu_jose_02_0011}
\delta_\infty X = \left\{ \hat{x}: \{x_n\}  \text{ is a Gromov sequence in } X\right\}.
\ee
Let $0<\epsilon<1$ and $o\in X$ be fixed. For $x, y \in X\cup \delta_\infty X$ we write
\be \label{tau}
\tau_{\epsilon, o}(x, y)= e^{-\epsilon(x|y)_o}.
\ee
The function $\tau_{\epsilon, o}$ need not be a metric but V{\"a}is{\"a}l{\"a} \cite[Proposition 5.16]{vaisala_2004_gh} showed that if $\epsilon\delta<1/5$, then there exists a metametric $\theta_{\epsilon, o}$ in $X\cup \delta_\infty X$ satisfying 
\be \label{metametric}
\tau_{\epsilon, o}(x, y) /2 \leq \theta_{\epsilon, o}(x, y) \leq \tau_{\epsilon, o}(x, y)
\ee
for all $x, y \in X\cup \delta_\infty X$. Moreover, it was shown that $ \theta_{\epsilon, o}(x, y) $ is a metric on $\delta_\infty X$.\\[2mm]
In a general hyperbolic space, there are no geodesic rays joining a point in the space to the boundary. Inorder to overcome this difficulty, V{\"a}is{\"a}l{\"a} introduced the theory of roads.
\begin{definition}[Roads]
Let $X$ be a metric space and let $\mu \geq 0$, $h\geq 0$. A $(\mu, h)-$road in $X$ is a sequence $\bar{\alpha}$ of arcs $\alpha_i:y\cra u_i$ with the following properties.
\begin{enumerate}
\item Each $\alpha_i$ is $h-$short.
\item The sequence of lengths $l(\alpha_i)$ is increasing and tends to $\infty$.
\item For $i\leq j$, the length map $g_{ij}:\alpha_i \rightarrow \alpha_j$ with $g_{ij}y=y$ satisfies $|g_{ij}x-x|\leq \mu $ for all $x\in \alpha_i$.
\end{enumerate}
\end{definition}
V{\"a}is{\"a}l{\"a} obtained the following result.
\begin{theorem}
Let $X$ be an intrinsic $\delta-$hyperbolic spaace and let $y\in X$, $\eta \in \delta_\infty X$, $h>0$. Then there is a $(4\delta+2h, h)$-road $\bar{\alpha}:y \cra \eta$.
\end{theorem}
Also, we point out that if $\bar{\alpha}$ is a $(\mu, h)$-road then one can construct a map $\psi :[0, \infty) \rightarrow X$ (see \cite[Remark 6.4]{vaisala_2004_gh}) that satisfies the rough isometry condition
$$|s-t|-\mu-h\leq |\psi(s)-\psi(t)| \leq |s-t|+\mu .$$
Gehring and Hayman \cite{gehring_hayman_} have shown that the hyperbolic geodesics in a simply connected hyperbolic domain in Euclidean spaces are essentially the shortest curves connecting any two given points upto a universal multiplicative constant. 
Bonk, Heinonen, and Koskela \cite{BHK}, have established a Gehring-Hayman theorem for conformal deformations of a geodesic Gromov hyperbolic metric space.
We \cite{allu_jose_} were able to show that the $h-$short arcs can be considered as a substitute for geodesics if we choose $h>0$ sufficiently small and obtained the following theorem analogous to  \cite[Theorem 5.1]{BHK}. 
Note that in \cite{allu_jose_} we have obtained the result for $(\kappa, h)$-Rips space which is an equivalent definition for $\delta$ hyperbolic space, (see \cite[Theorem 2.34, Theorem 2.35]{vaisala_2004_gh}). 
We recall a slightly modified version of \cite[Theorem 1.3]{allu_jose_}.
\begin{theorem}\label{allu_jose_02_0002}
Let $X$ be an intrinsic metric space which is $\delta-$hyperbolic and $\rho :X\ra (0, \infty)$ be a continuous function that satisfies 
\be \label{allu_jose_02_0003}
\exp\left(-\epsilon|x-y|\right) \leq \frac{\rho(x)}{\rho(y)} \leq \exp\left(\epsilon|x-y|\right)
\ee
for all $x,y \in X$ and  $\epsilon >0$. For $h<1/13$, let $\mathcal{A}$ denote the family of all $h-$short arcs $\alpha:x\cra y$ for any $x, y\in X$ satisfying $l(\alpha) \leq 2|x-y|$.
Then, there exists $\epsilon_0(\delta, h)>0$, $K(\delta, h)>0$ such that for any $0<\epsilon\leq \epsilon_0$  and $x, y\in X$ 
\begin{equation}\label{allu_jose_02_0004}
l_\rho(\alpha) \leq K l_\rho (\gamma)
\end{equation}
where  $\gamma: x\cra y $ is any curve and  $\alpha\in \mathcal{A}$ with same endpoints.
\end{theorem}
The following lemma from \cite{allu_jose_} is essential to our main result.
\begin{lemma}\label{lem1}
Let $\gamma:x\cra y$ be an $h-$short arc in the metric space $X$ and $p\in X$ be fixed. Let $x_\gamma, y_\gamma\in \gamma$ be points in the curve $\gamma$ such that $l\left(\gamma[x, x_\gamma] \right)= (y|p)_x$ and $l\left(\gamma[y_\gamma, y] \right)= (x|p)_y$. If $z\in \gamma[x, y_\gamma]$ and $u\in \gamma[x, z]$, then
$$|p-u|-|p-z|\geq |u-z|-8\delta-8h.$$
\end{lemma}
\section{Uniformization}
To show that the deformed space is incomplete we need the following crucial lemma.
\begin{lemma}\label{non-completeness lemma}
Let $X$ be an intrinsic $\delta-$hyperbolic metric space and $\bar{\alpha} : o \cra \zeta,\,\, \alpha_i :o \cra u_i $ be a $(\mu, h)$-road, where $o\in X$ and $\zeta \in \delta_\infty X$. 
Then, for $n\leq m $, there exists a continuous $(1, K)$-rough quasi-geodesic from $o$ to $u_m$ that passes through $u_n$, where the constant $K$ depends only on $\mu$ and $h$.
\end{lemma}
\begin{proof}
Let $g_{nm}:\alpha_n \rightarrow \alpha_m $ be the length map with $|g_{nm}x-x| \leq \mu $ for all $x\in \alpha_n$. Choose an $h-$short arc $\beta : u_n \cra g_{nm}u_n. $ Set $\lambda = l(\beta) $. Then, $\lambda  \leq \mu + h$. Further, we assume that $\alpha_n:[0, L_n] \ra X, \alpha_m:[0, L_m] \ra X$ and $\beta :[0, \lambda] \ra X$ are the arc length parametrizations of the curves. We define arc $\gamma: o\cra u_m $ as the concatenation of the curves $\alpha_n, \beta $ and $\alpha_m[L_n, L_m]$.\\[2mm]
That is, define $\gamma : [0, L_m+\lambda]\ra X $ as follows.
\[\gamma(t)=
\begin{cases}
\alpha_n(t), & 0\leq t < L_n\\
\beta(t-L_n), & L_n\leq t < L_n+\lambda\\
\alpha_m(t-\lambda),&  L_n+\lambda \leq t \leq L_m+\lambda.
\end{cases}
\]
Note that $\gamma $ is continuous as $\beta(\lambda) = g_{nm}u_n = \alpha_m (L_n)$, since $g_{nm}$ is a length map. Recall that the arclength parametrization of an $h-$short arc satisfies rough quasi geodesic condition (see Definition \ref{hroughgeo}).\\[2mm]
To show that the curve $\gamma$ is $(1, K)$-rough quasi-geodesic  we need to show that
\[
|s-t|-K \leq |\gamma(s) -\gamma(t)| \leq |s-t|+K.
\]
But, since $\gamma$ is arclength parametrized we always have 
\[|\gamma(s) -\gamma(t)| \leq l(\gamma(s, t)) \leq |s-t|.\]
Therefore, it suffices to show that  $|\gamma(s) -\gamma(t)|\geq  t-s-K $ whenever   $s\leq t $. We consider three cases.\\[2mm]
\textit{Case 1:} $s\in [0, L_n), t\in [L_n, L_n+\lambda)$. Then,
\begin{align*}
|\gamma(s)-\gamma(t)|& = |\alpha_n(s) - \beta(t-L_n)|\\
&\geq |\alpha_n(s) - \alpha_n(L_n) | -|\alpha_n(L_n)-\alpha_m(L_n)|- | \alpha_m(L_n) - \beta(t-L_n)|\\
& \geq L_n-s-h-\mu -\lambda\\
&\geq t-s -\mu -h-\lambda +L_n -t \\
&\geq t-s -\mu -h-2\lambda \\
&\geq  t-s-3\mu -3h,
\end{align*}
where in the last step we used the fact that  $\lambda  \leq \mu + h$.\\[2mm]
\textit{Case 2:} $s\in [0, L_n)$ and $t\in [L_n+\lambda, L_m+\lambda]$. Then,
\begin{align*}
|\gamma(s)-\gamma(t)|&\geq |\alpha_m(s)-\alpha_m(t-\lambda)|- |\alpha_n(s) -\alpha_m(s)|\\
&\geq l\left(\alpha_m\left(s, t-\lambda\right)\right) -h -\mu \\
&\geq t-\lambda-s-h-\mu \\
&\geq t-s -2\mu-2h .
\end{align*}
\textit{Case 3:} $s\in [L_n, L_n+\lambda) $ and $t\in [L_n+\lambda, L_m+\lambda]$. We compute,
\begin{align*}
|\gamma(s) -\gamma(t)| &= |\beta(s-L_n) - \alpha_m(t-\lambda)|\\
&\geq |\alpha_m(L_n)-\alpha_m(t-\lambda)| -|\alpha_m(L_n) - \beta(0)| -|\beta(s-L_n) -\beta (0)|\\
&\geq t-\lambda-L_n-h-\mu -s+L_n \\
&\geq t-s -2\mu-2h.
\end{align*}
In all the other cases, it follows from the rough quasi geodesic property of the $h-$short arcs that
$$|s-t|-h \leq |\gamma(s)-\gamma(t)| \leq |s-t|.$$
Therefore,in all cases we obtain 
$$ t-s-3\mu -3h \leq |\gamma(s) -\gamma(t)| \leq |s-t|.$$
Thus, the curve $\gamma$ is a $(1, K)$ rough quasi-geodesic with $K= 3\mu +3h$. Hence, the lemma follows.
\end{proof}

\begin{proof}[Proof of Theorem \ref{uniformization}]
We first prove that $X_\epsilon$ is a bounded metric space. Let $x\in X$ and let $\alpha$ be a $1-$short arc joining $x$ and $p$, further assume that $\alpha $ is parametrized by arc length. Then,
$$d_\epsilon(x, p) \leq \int_\alpha \rho_\epsilon \ds \leq \int_0^\infty \exp\left(-\epsilon \left(l\left(\alpha[\alpha(t), p]\right)-1\right)\right) \dt = e^\epsilon \int_0^\infty e^{-\epsilon t} \dt = \frac{e^\epsilon}{\epsilon} ,$$
which yields us 
$$\diam X_\epsilon \leq \frac{2e^\epsilon}{\epsilon},$$
and hence $X_\epsilon$ is bounded.\\[2mm]
Our next aim is to prove that $X_\epsilon$ is a rectifiably connected metric space. 
To this end, we intend to show that the identity map $\left(X, |x-y|\right)\rightarrow \left(X, d_\epsilon\right)$ is locally bilipschitz.  
Let $w\in X$ be fixed and consider the open ball $B(w, 1)$. Then, for every $z\in B(w, 1)$, by virtue of \eqref{harnack}, we obtain
$$e^{-\epsilon}\rho_\epsilon (w)\leq \rho_\epsilon (z) \leq e^{\epsilon}\rho_\epsilon(w).$$
Now, for $x, y\in B(w, 1)$, for an arc   $\alpha: x\cra y$ satisfying $l(\alpha)\leq 2|x-y|$, and  for any $u\in \alpha$ we have by \eqref{harnack} that,
$$\rho_\epsilon(u) \leq \exp\left(\epsilon |x-u|\right)\rho_\epsilon(x)\leq e^{4\epsilon}\rho_\epsilon(x) \leq e^{4\epsilon+\epsilon} \rho_\epsilon (w).$$
Thus, for $x, y\in B(w, 1)$, this implies that,
$$d_\epsilon(x, y) \leq \int_\alpha \rho_\epsilon \ds \leq 2e^{4\epsilon+\epsilon} \rho_\epsilon(w) |x-y|.$$
Now, for any curve $\gamma$ joining $x$ and $y$ we can find a subcurve $\gamma'$ of $\gamma$ emanating from $x$ such that $l(\gamma')=|x-y|$. Therefore, by \eqref{harnack} one computes,
\begin{align*}
d_\epsilon(x, y) &= \inf_\gamma\int_\gamma \rho_\epsilon \ds \geq \inf_\gamma\int_{\gamma'} \rho_\epsilon \ds\\
&\geq \int_0^{|x-y|}\rho_\epsilon(x) e^{-\epsilon t}\dt\\
&\geq \rho_\epsilon(x) e^{-4\epsilon}|x-y|\\
&\geq e^{-4\epsilon-\epsilon}\rho_\epsilon(w).
\end{align*}
Hence, we see that the identity map is locally bilipschitz and thus $X_\epsilon$ is a rectifiably connected metric space. \\[2mm]
Next, we show that $X_\epsilon$ is incomplete. Let $\zeta \in \delta_\infty X $ and $\bar{\alpha}: p \cra \zeta$; $\alpha_n:w\cra u_n$, be a $(4\delta+2h, h)-$road, where $h>0$. We wish to show that $\left\{u_n\right\}$ is a Cauchy sequence in the metric space $(X_\epsilon, d_\epsilon)$. For $n\leq m $ by Lemma \ref{non-completeness lemma}, we obtain a continuous $(1, K)$- rough quasi-geodesic $\gamma: p\cra u_m$  such that $u_n\in \gamma$, where $K$ depends only on $\delta$ and $h$. It is clear from the construction that $\gamma$ is arclength parametrized. We assume that $\gamma:[0, T] \rightarrow X$ is the parametrization with $\gamma(0)= p, \gamma(T)=u_m$ and $\gamma(t_n)=u_n$ where $t_n\in (0, T)$. \\
Now, we compute
\begin{align*}
d_\epsilon(u_n, u_m) &\leq \int_{\gamma|_{[t_n, T]}} \rho_\epsilon ds \\
&\leq \int_{t_n}^T e^{-\epsilon t+ \epsilon K} dt\\
&\leq \frac{1}{\epsilon} e^{\epsilon K } e^{-\epsilon t_n}.
\end{align*}
It is evident from the construction of the rough quasigeodesic, that $t_n= l(\alpha_n) \rightarrow \infty $ as $n \rightarrow \infty $.  Thus, it follows that $\{u_n\} $ is $d_\epsilon$-Cauchy. Since the spaces $X$ and $X_\epsilon$ are locally bilipschitz, it is easy to see that the sequence $\{u_n\} $ cannot converge to a point in $X_\epsilon$. Hence, it follows that $X_\epsilon$ is incomplete.
\\[2mm]
For any $y\in X$ satisfying $|x-y|>\frac{1}{\epsilon}$, and for any rectifiable curve $\gamma$ joining them we have a subcurve $\gamma'$ of $\gamma$ originating from $x$ with $l(\gamma')=1/\epsilon$. Then by \eqref{harnack} we compute,
\begin{align*}
d_\epsilon(x, y) &= \inf_\gamma \int_\gamma \rho_\epsilon \ds \geq \inf_\gamma \int_{\gamma'}  \rho_\epsilon \ds\\
& \geq \int_0^{\frac{1}{\epsilon}} \rho_\epsilon(x) e^{-\epsilon t} \dt \geq \left( 1-e^{-1} \right)\frac{\rho_\epsilon(x)}{\epsilon}\\
&\geq \frac{\rho_\epsilon(x)}{e\epsilon}.
\end{align*}
Therefore, for each $x\in X$, we find that 
\begin{equation}\label{bdry}
d_\epsilon(x) = d_\epsilon\left(x, \partial_\epsilon X\right) \geq \frac{\rho_\epsilon(x)}{e\epsilon},
\end{equation}
where $\partial_\epsilon X = \partial X_\epsilon = \bar{X_\epsilon}\setminus X_\epsilon.$\\[2mm]
We are now equipped to prove that $X_\epsilon$ is indeed a uniform space. The quasiconvexity of the arc follows from Theorem \ref{allu_jose_02_0002}. 
Therefore, it suffices to prove that every $h-$short arc satisfies the double cone arc condition. Let $x, y\in X_\epsilon$ and let $\gamma:x\cra y$ be an $h-$short arc joining them. 
Without loss of generality, we may suppose that $z\in \gamma_a\cup \gamma*= \gamma[x, y_\gamma]$, then by Lemma \ref{lem1}, we have
 $$|p-u|\geq |p-z|+|u-z|-8\delta-8h,$$
 for every $u\in \gamma[x, z]$.
Therefore, by \eqref{bdry} we obtain
\begin{align*}
l_\epsilon\left(\gamma\right) &= \int_{\gamma[x, z]} \rho_\epsilon(u) |du|\\
&= \int_{\gamma[x, z]} \exp\left(-\epsilon|p-u|\right) |du| \\
&\leq \exp\left(\delta\epsilon+8h\epsilon\right) \exp \left(-\epsilon|p-z|\right) \int_{\gamma[x, z]} \exp \left(-\epsilon|u-z|\right) |du|\\
&\leq \exp\left(\delta\epsilon+8h\epsilon\right) \rho_\epsilon(z) \int_0^{\infty} e^{h\epsilon} e^{-\epsilon t} dt\\
&= \frac{\exp\left(\delta\epsilon+8h\epsilon\right)}{\epsilon} \rho_\epsilon(z) \\
& \leq \exp\left(\delta\epsilon+9h\epsilon+1\right) d_\epsilon(z),
\end{align*}
which completes the proof.
\end{proof}
To show there exists a natural boundary identification between the Gromov boundary and the metric boundary we establish the following crucial lemma.
\begin{lemma}\label{12}
Let $X$ be an intrinsic $\delta-$hyperbolic space and let $X_\epsilon$ be its uniformization for $0<\epsilon \leq \epsilon_0$. Then there is a constant $C\geq 1$ such that 
\be \label{12eq}
\frac{1}{C}d_\epsilon(x, y) \leq \frac{e^{-\epsilon(x|y)_p}}{\epsilon} \min \left(1/2, \epsilon |x-y|\right) \leq C d_\epsilon(x, y)
\ee
for every $x, y \in X$.
\end{lemma}
\begin{proof}
For $x, y\in X$ let $\gamma : x\cra y$ be an $h-$short arc and let $x_\gamma$ and $y_\gamma$ be the points in $\gamma$ induced by $p\in X$. 
Further, choose an $h-$ short arc $\alpha : x\cra p $ and $x_\alpha$ and $p_\alpha$ be the points induced by $y\in X$. \\
By invoking Tripod Lemma \ref{allu_jose_02_0013}, we obtain
$$|x_\gamma - x_\alpha| \leq 4\delta +2h.$$
Thus,
\begin{align*}
|x_\gamma - p | &\leq |x_\gamma-x_\alpha|+ |x_\alpha-p|\\
&\leq 4\delta +2h + l\left(\alpha[x_\alpha, p_\alpha]\right) + l\left(\alpha[p_\alpha, p]\right)\\
&\leq  4\delta+3h +(x|y)_p,
\end{align*}
which yields us
\be \label{12.1}
|x_\gamma - p |-(x|y)_p \leq 4\delta+3h .
\ee
Now, 
\begin{align*}
(x|y)_p  &= l\left(\alpha[p_\alpha, p]\right)\\
& \leq |p_\alpha-p|+h\\
&\leq |p-x_\gamma| +|x_\gamma - x_\alpha|+|x_\alpha-p_\alpha|+h\\
&\leq |p-x_\gamma|+4\delta+4h
\end{align*}
which implies that
\be \label{12.2}
(x|y)_p-|p-x_\gamma| \leq 4\delta+4h.
\ee
Therefore, \eqref{12.1} and \eqref{12.2} together yields us
\be \label{12.3}
\left| |p-x_\gamma| - (x|y)_p\right| \leq 4\delta+4h.
\ee
Furthermore, we assume that $\gamma: x\cra y$ is an $h-$short arc satisfying $l(\gamma)\leq 2|x-y|$. Note that we can always choose such a curve for any $h>0$. 
To see that, let $$c=\min\left\{ 2, 1+\frac{h}{|x-y|}\right\}.$$ 
Then clearly $c>1$. Since, the space is intrinsic one always can choose a curve $\gamma: x\cra y$ such that $l(\gamma)\leq c |x-y|$. 
It is easy to verify that $\gamma$ is an $h-$short arc which satisfies $l(\gamma)\leq 2|x-y|$.\\[2mm]
Now, we consider two cases. \\[2mm]
If $\epsilon|x-y| \leq \frac{1}{2}$, then for every $u\in \gamma$ we have
\be \label{12.4}
\epsilon|x_\gamma-u| \leq \epsilon l(\gamma) \leq 2\epsilon |x-y| \leq 1.
\ee
Therefore, by \eqref{harnack} and \eqref{12.4}  we have
\begin{align*}
\frac{1}{e}\rho_\epsilon(x_\gamma) &\leq e^{-|x_\gamma - u|} \rho_\epsilon(x_\gamma) \\
&\leq \rho_\epsilon(u) \leq e^{\epsilon|x_\gamma -u|} \rho_\epsilon(x_\gamma) \\&\leq e \rho_\epsilon(x_\gamma).
\end{align*}
This implies that
$$d_\epsilon(x, y) \leq \int_\gamma \rho_\epsilon(u) |du| \leq 2 e \rho_\epsilon(x_\gamma) |x-y|.$$ 
To obtain the reverse inequality of \eqref{12eq} we make use of the Gehring-Hayman theorem (Theorem \ref{allu_jose_02_0002}), for that we may assume that $\gamma:x\cra y$ is an $h-$short arc with $h<\frac{1}{13}$. 

Therefore, we obtain
$$d_\epsilon(x, y) \geq \frac{1}{K} \int_\gamma \rho_\epsilon(u) |du| \geq \frac{1}{eK}\rho_\epsilon(x_\gamma) |x-y|,$$
where the constant $K$ depends only on $\delta$ and $h$. One can always choose $h=1/14$ to obtain the constant $K$ as depending only on $\delta$. Hence in this case the assertion of Lemma \ref{12} holds.\\[2mm]
Now, if $\epsilon|x-y|>1/2$, then by Lemma \ref{lem1}, we have
$$|p-u|\geq |p-z|+|z-u|-8\delta-8h$$
for all $x\in \gamma [x, x\gamma]$ and $u\in \gamma[x, z]$. In particular, for $z=x_\gamma$ we have 
$$|p-u|\geq |p-x_\gamma|+|x_\gamma-u|-8\delta-8h$$
for all $u\in \gamma[x, x_\gamma]$. 
Also, by symmetry, we have
$$|p-u|\geq |p-x_\gamma|+|x_\gamma -u|-8\delta-8h$$
for all $u\in \gamma[x_\gamma, y]$.\\
Therefore, we conclude that
$$|p-u|\geq |p-x_\gamma|+|x_\gamma-u|-8\delta-8h$$
for all $u\in \gamma$.
Thus, for all $u\in \gamma$, we obtain
$$e^{-\epsilon|p-u|}\leq e^{8\delta \epsilon} e^{8h\epsilon} \rho_\epsilon(x_\gamma) e^{-\epsilon |x_\gamma-u|}.$$ 
For easy computation, we use a different interval for the arc length parametrization for the arc $\gamma:x\cra y$. 
Let $L=l(\gamma)$ and suppose that $\gamma:[0, L] \rightarrow X$ is the arclength parametrization of $\gamma$. 
Let $t_0\in [0, L]$ be such that $\gamma(t_0)=x_\gamma$. 
Then, it is possible to have a parametrization of the arc $\gamma$ from the interval $[-t_0, L-t_0]$ such that $0$ maps to $x_\gamma$. For simplicity, we also denote this parametrization by $\gamma$.
Now, $\gamma : [-t_0, L-t_0] \rightarrow X$ is an arclength parametrization of the arc such that $\gamma(0)=x_\gamma$.\\[2mm]
Now, we compute
\begin{align*}
d_\epsilon(x, y) &\leq \int_\gamma \rho_\epsilon|du|\\
&\leq e^{8\delta\epsilon}e^{8h\epsilon}\rho_\epsilon(x_\gamma)\int_\gamma e^{-\epsilon |x_\gamma-u|} |du|\\
&\leq 2 e^{8\delta\epsilon}e^{8h\epsilon}\rho_\epsilon(x_\gamma)\int_{\gamma[x_\gamma, y]} e^{-\epsilon |x_\gamma-u|} |du|\\
&\leq 2 e^{8\delta\epsilon}e^{8h\epsilon}\rho_\epsilon(x_\gamma)\int_{\gamma[x_\gamma, y]} e^{-\epsilon \left(l(\gamma[x_\gamma, u]-h\right)} |du|\\
&\leq 2 e^{8\delta\epsilon}e^{9h\epsilon}\rho_\epsilon(x_\gamma)\int_0^\infty e^{-\epsilon t} dt\\
&\leq \frac{2}{\epsilon} e^{8\delta\epsilon}e^{9h\epsilon}\rho_\epsilon(x_\gamma)
\end{align*}
On the other hand, we have
$$|p-u|\leq |p-x_\gamma|+|x_\gamma-u|,$$
thereby we obtain 
$$e^{-\epsilon|p-u|} \geq \rho_\epsilon(x_\gamma) e^{\epsilon|x_\gamma-u|}$$
for every $u\in \gamma$.
Now, without loss of generality we may assume that $h<1/13$ and thus, by employing Gehring-Hayman theorem (Theorem \ref{allu_jose_02_0002}), we obtain
\begin{align*}
d_\epsilon(x, y) &\geq \frac{1}{K}\int_\gamma\rho_\epsilon|du|\\
&\geq \frac{1}{K} \rho_\epsilon(x_\gamma) \int_\gamma e^{-\epsilon |x_\gamma-u|} |du|\\
&\geq \frac{1}{K} \rho_\epsilon(x_\gamma) \int_0^{|x-y|/2} e^{-\epsilon t} dt\\
&\geq \frac{\rho_\epsilon(x_\gamma)}{K\epsilon} \left(1-e^{-1/4}\right).
\end{align*}
Thus, the lemma holds.
\end{proof}
\begin{proof}[Proof of Theorem \ref{thm2}]
Let $\{u_n\}$ be a Gromov sequence in $X$. Then, $\left(u_n|u_m\right)_p \rightarrow \infty $ as $n, m\rightarrow\infty$. 
Thus, it follows from \eqref{12eq} that $\{u_n\}$ is $d_\epsilon$-Cauchy. 
Now, since the identity map between $X$ and $X_\epsilon$ is locally bilipschitz, it is easy to see that $\{u_n\}$ cannot converge to a point in $X_\epsilon$.
Therefore, $\{u_n\}$ converges to a point in $\delta_\epsilon X$. 
Moreover, if $\{u_n\}$ and $\{v_n\}$ are two equivalent Gromov sequences, then by definition we have, $(u_n|v_n)_p \rightarrow \infty$. 
Therefore, an application of \eqref{12eq} shows that $\{u_n\}$ and $\{v_n\}$ are equivalent Cauchy sequences in $X_\epsilon$, \textit{i.e.,} $d_\epsilon(u_n, v_n) \rightarrow 0$ as $n\rightarrow \infty$.  
Hence, the map which sends an equivalence class of Gromov sequences to the point in $\delta_\epsilon X$ to which the sequence convergence in the $d_\epsilon$-metric is well defined. \\[2mm]
To see that the map is injective we make use of \eqref{12eq} again. Suppose, $\{u_n\}$ and $\{v_n\}$  are two sequences whose image under the map is same, which means that $\{u_n\}$ and $\{v_n\}$ are equivalent Cauchy sequences in the metric $d_\epsilon$.
Thus, \eqref{12eq} again implies that $$e^{-\epsilon(u_n|v_n)_p} \min \left(1/2, \epsilon |u_n-v_n|\right)\rightarrow 0.$$
If $e^{-\epsilon(u_n|v_n)_p}\rightarrow 0$, then it would imply that $(u_n|v_n)_p \rightarrow \infty$ and hence $\{u_n\}$ and $\{v_n\}$ are equivalent Gromov sequences and belong to the same equivalence class.
If $|u_n-v_n|\rightarrow 0$ then from the definition of $(u_n|v_n)_p$ it follows that $(u_n|v_n)_p \rightarrow \infty.$ 
Thus, in both the cases we see that no two non-equivalent Gromov sequences can be send to same point under the map. Therefore, the map is injective.
By following a similar argument as above, one can see that the map is surjective as well.
Therefore, we obtain a natural identification of the Gromov boundary of $X$ and the metric boundary of $X_\epsilon$.
\\[2mm]
For $x, y\in \delta_\infty X$ we define $|x-y|=0$ if $x= y$ and $|x-y|=\infty $
if $x\neq y$. Then, it follows from \eqref{12eq}, \eqref{tau}, and \eqref{metametric} that the map is a quasi-isometry from $\left(\delta_\infty X, \theta_{\epsilon, o}\right) $ to $\left(\delta_\epsilon X, d_\epsilon\right)$.
\end{proof}

\noindent{\bf Acknowledgement.}
The first author thanks SERB-CRG and the second author's research work is supported by CSIR-UGC.\\
\noindent\textbf{Compliance of Ethical Standards:}\\
\noindent\textbf{Conflict of interest.} The authors declare that there is no conflict  of interest regarding the publication of this paper.\vspace{1.5mm}\\
\noindent\textbf{Data availability statement.}  Data sharing is not applicable to this article as no datasets were generated or analyzed during the current study.\vspace{1.5mm}\\
\noindent\textbf{Authors contributions.} Both the authors have made equal contributions in reading, writing, and preparing the manuscript.

\end{document}